\documentclass{article}

\usepackage{amsmath,yhmath,amsfonts,amssymb,amsthm,bbding}
\usepackage{mathrsfs}

\newcommand{\sn}{{\operatorname{sn}}}

\newcommand{\counte}{theorem}

\newtheorem{lemma}[\counte]{\bf Lemma}

\newtheorem{coro}[\counte]{\bf Corollary}

\newtheorem{remark}[\counte]{\bf Remark}

\newtheorem{prop*}{\bf Step} 

\allowdisplaybreaks

\numberwithin{equation}{section}

\renewcommand{\thefootnote}{\fnsymbol{footnote}}

\begin{document}

\renewcommand{\thefootnote}{\arabic{footnote}}

\centerline{\bf\Large A Schur's type volume comparison theorem
\footnote{Supported by NSFC 12371050. \hfill{$\,$}}}

\vskip5mm

\centerline{Xiaole Su, Yi Tan, Yusheng Wang\footnote{The
corresponding author (E-mail: wyusheng@bnu.edu.cn). \hfill{$\,$}}}

\vskip6mm

\noindent{\bf Abstract.} In this paper, inspired by Schur's comparison theorem  about curves in Euclidean space,
we mainly provide a Schur's type volume comparison theorem, which is about the volumes of the boundaries of open balls in a complete $n$-dimensional Riemannian manifold with Ricci$\geq (n-1)k$.

\vskip1mm


\vskip1mm


\vskip6mm

\setcounter{section}{0}


\section{Main results}

In differential geometry, a basic and interesting result is Schur's comparison theorem ([Ho]) which asserts that: {\it Given two arc-length parameterized $C^2$-curves $\gamma_i(s)|_{[0,l]}$ ($i=1,2$) in Euclidean space $\Bbb R^3$, if $\gamma_1(s)|_{[0,l]}$ together with the chord $[\gamma_1(0)\gamma_1(l)]$ forms a convex and simple closed curve
and the curvature $|\gamma_1''(s)|\geq|\gamma_2''(s)|$ for all $s\in [0,l]$, then the distance
$$|\gamma_1(0)\gamma_1(l)|\leq|\gamma_2(0)\gamma_2(l)|;$$
and if the equality holds, then the two curves are congruent. }

Another significant result is the following relative volume comparison theorem (Bishop-Gromov),
where we denote by $\Bbb S_k^n$ the simply connected and complete $n$-dimensional space form of constant curvature $k$.

\vskip2mm

\noindent {\bf Theorem 1.1 ([Pe]).} {\it Let $M$ be a complete $n$-dimensional Riemannian manifold with ${\rm Ric}_M\geq (n-1)k$. Then for any $p\in M$ and $\tilde p\in \Bbb S_k^n$,
$$ \frac{\text{\rm Vol}(B(p, \rho))}{\text{\rm Vol}(B(\tilde p, \rho))}, \text{ whose limit is equal to $1$ as $\rho\to 0^+$, is decreasing with respect to $\rho$}.\eqno{(1.1)}$$}
\hskip4mm As a result of Theorem 1.1,  the volume {\rm Vol}$(B(p, \rho))\leq${\rm Vol}$(B(\tilde p, \rho))$;
moreover, we can show that if the equality holds then the open ball $B(p, \rho)$ is isometric to $B(\tilde p, \rho)$
(see Footnote 8 below). This is just the well known Bishop's volume comparison theorem ([Pe], [WSY]).

Note that Schur's comparison theorem just compares the ``sizes'' of the boundaries of two curves of the same length. Inspired by it, a natural question is whether we can compare
the ``sizes'' of the boundaries of $B(p, \bar r)$ and $B(\tilde p, r)$ when Vol$(B(p, \bar r))=$ Vol$(B(\tilde p, r))$ in Bishop's volume comparison theorem.
Our main result is just to answer the question.

\vskip2mm

\noindent {\bf Theorem A.} {\it Let $M$ be a complete $n$-dimensional Riemannian manifold with ${\rm Ric}_M\geq (n-1)k$. Given $p\in M$ and $\tilde p\in \Bbb S_k^n$, if {\rm Vol}$(B(p, \bar r))=${\rm Vol}$(B(\tilde p, r))$
with $r<\frac\pi{\sqrt{k}}$ for $k>0$, then
$$\text{{\rm Vol}$(\partial B(p, \bar r))\leq ${\rm Vol}$(\partial B(\tilde p, r))$ and {\rm Vol}$'_+(\partial B(p, \rho))|_{\rho=\bar r}\leq${\rm Vol}$'(\partial B(\tilde p, \rho))|_{\rho=r}$};$$
and if any of the two equalities holds, then $\bar r=r$. Moreover, if {\rm Vol}$(\partial B(p,\bar r))>0$, then $\bar r$ as a function with respect to $r$ is continuous, and $\bar r_+'\geq1$ at $r$ and the equality implies $\bar r=r$.}

\vskip2mm

For the definition of {\rm Vol}$(\partial B(p,\rho))$ refer to Section 3 below. In any rigidity case where $\bar r=r$ in Theorem A,  $B(p, r)$ is isometric to $B(\tilde p, r)$ by Bishop's volume comparison theorem.
And if $k>0$ and $r=\frac\pi{\sqrt{k}}$ in Theorem A, then $M$ has to be isometric to  $\Bbb S_k^n$ by the Maximal Volume Theorem ([Pe]).

In Theorem A, it is clear that $\bar r\geq r$ by Bishop's volume comparison theorem. We would like to point out that we can also compare
the ``sizes'' of the boundaries of $B(p, r)$ and $B(\tilde p, r)$. From the proof of Theorem 1.1, we can in fact conclude that
$\text{Vol$(\partial B(p, r))\leq $Vol$(\partial B(\tilde p, r))$}$, and $\text{Vol$_+'(\partial B(p, r))\leq$Vol$'(\partial B(\tilde p, r))$}$
(see Footnotes 6 and 7 below).

Note that ${\rm Vol}(\partial B(\tilde p, r))$ in Theorem A with $r\leq\frac{\pi}{2\sqrt k}$ for $k>0$
is increasing with respect to $r$, then it is easy to see that:

\vskip2mm

\noindent {\bf Corollary B.} {\it Let $M$ be a complete $n$-dimensional Riemannian manifold with ${\rm Ric}_M\geq (n-1)k$. Given $p\in M$ and $\tilde p\in \Bbb S_k^n$, if {\rm Vol}$(\partial B(p, \bar r)=${\rm Vol}$(\partial B(\tilde p, r))$ with $r\leq\frac{\pi}{2\sqrt k}$ for $k>0$, then {\rm Vol}$(B(p, \bar r))\geq${\rm Vol}$(B(\tilde p, r))$; and if the equality holds, then $\bar r=r$ and thus
$B(p, r)$ is isometric to $B(\tilde p, r)$.}

\vskip2mm

Moreover, similar to the monotonicity in (1.1), we have the following observation.

\vskip2mm

\noindent {\bf Theorem C.} {\it For the $p$ and $\tilde p$ in Theorem A, if {\rm Vol}$(B(p, \bar r))=${\rm Vol}$(B(\tilde p, r))$, then
$$ \text{ both  $\frac{r}{\bar r}$ and }  \frac{\text{\rm Vol}(\partial B(p, \bar r))}{\text{\rm Vol}(\partial B(\tilde p, r))} \text{ are decreasing with respect to $r$} \eqno{(1.2)}$$
in any of the following two cases:

{\rm (C1)}\ \ $n=2$, $k\geq 0$, and $\bar r$ is less than the injective radius of $M$ at $p$;

{\rm (C2)}\ \ $M$ is isometric to $\Bbb S^n_{\bar k}$ with $\bar k\geq k$.}

\vskip2mm

In other cases, we cannot see the monotonicity in (1.2) so far.

\vskip2mm

In the rest of the paper, we will first introduce two lemmas on real functions  in Section 2 used in the proofs of Theorems A and C.
Then we will give proofs of Theorems A and C in Section 3. At last, as an application of Lemma 2.2, we will give a relative version of the classical Toponogov's Theorem.


\section{Two lemmas on real functions}

It turns out that Theorems A and C can be reduced to the following properties of real functions.
We will denote by $\text{\rm sn}_k(t)$ the functions $\frac{1}{\sqrt k}\sin(\sqrt{k}t)$, $t$,  $\frac{1}{\sqrt{-k}}\sinh(\sqrt{-k}t)$
for $k>0,\ =0,\ <0$ respectively.

\begin{lemma} Let $f: [0,l]\to \mathbb{R}$ be a non-negative and Lebesgue integrable function with $f(0)=0$ and $D^+f(0)=1$.
Suppose that $\frac{f(t)}{\text{\rm sn}_k(t)}$ is decreasing in $(0,l]$ with $l<\frac{\pi}{\sqrt{k}}$ for $k>0$.
Given an integer $m>0$, if $\int_0^{x}f^m(t)dt=\int_0^{r}{\text{\rm sn}_k^m(t)}dt$ for some $x,r\in[0,l]$, then the followings hold:

\vskip1mm

{\rm(2.1.1)} $f(t)\leq\text{\rm sn}_k(t)$ on $[0,l]$, and so $x\geq r$ and the equality implies $f(t)={\text{\rm sn}_k(t)}$ on $[0,r]$. And if $f(t_0)>0$,
then we can define a continuous function $x(r)$ by $r\mapsto x\in [0,t_0]$.

\vskip1mm

{\rm(2.1.2)} $f(x)\leq\text{\rm sn}_k(r)$, and the equality implies $x=r$ and so $f(t)={\text{\rm sn}_k(t)}$ on $[0,r]$.

\vskip1mm

{\rm(2.1.3)} $D^{\pm}f(x)\leq\text{\rm sn}'_k(r)$, and the equality implies $x=r$ and so $f(t)={\text{\rm sn}_k(t)}$ on $[0,r]$.

\vskip1mm

{\rm(2.1.4)} If in addition $f$ is right continuous, then the function $x(r)$ in (2.1.1) satisfies $x_{+}'(r)\geq 1$;
and if $x_{+}'(r_0)=1$, then $x(r_0)=r_0$ and so $f(t)={\text{\rm sn}_k(t)}$ on $[0,r_0]$.
\end{lemma}

Recall that, for a function $f: [a,b]\to \mathbb{R}$ (maybe not continuous), its upper and lower right Dini derivatives (cf. [KK]) at $x\in[a,b]$ are defined by
\begin{align*}
 D^{+}  f(x)\triangleq\varlimsup_{h \to 0^+} \frac{f(x+h)-f(x)}{h} \text{\ \ and\ \ }
 D_{+}  f(x)\triangleq\varliminf_{h \to 0^+} \frac{f(x+h)-f(x)}{h},
\end{align*}
and $D^{-}f(x), D_{-}f(x)$ are similarly defined as $h \to 0^-$.
Clearly, if $f$ is right differentiable at $x$, then its right Dini derivatives are both equal to $f'_+(x)$.

\begin{proof} (2.1.1)\ \
Since $f(t)\geq 0$ with $f(0)=0$, $D^+f(0)=1$ and $\frac{f(t)}{\text{\rm sn}_k(t)}$ is decreasing on $(0,l]$,
it's clear that $f(t)\leq\text{\rm sn}_k(t)$ on $[0,l]$. Then if $\int_0^{x}f^m(t)dt=\int_0^{r}{\text{\rm sn}_k^m(t)}dt$,
we can get $x\geq r$, and the equality implies $f(t)={\text{\rm sn}_k(t)}$ on $[0,r]$.
Moreover, if $f(t_0)>0$, then the decreasing property of $\frac{f(t)}{\text{\rm sn}_k(t)}$ implies that $f(t)>0$ on $(0, t_0)$ and we can define a continuous function $x(r)$ by $r\mapsto x\in [0,t_0]$.

\vskip1mm

(2.1.2)\ \  Note that we just need to consider $r\in(0,\frac{\pi}{2\sqrt{k}})$ if $k>0$ (in fact, by (2.1.1), $x\geq r$ and $f(x)\leq\text{\rm sn}_k(x)$; so if $r\geq\frac{\pi}{2\sqrt{k}}$, then $f(x)\leq\text{\rm sn}_k(r)$ and the equality implies $x=r$).
(2.1.2) is obviously true if there is no $t\in [r,l]$ such that $f(t)\geq\text{\rm sn}_k(r)$;
otherwise, by the decreasing property of $\frac{f(t)}{\text{\rm sn}_k(t)}$
there is $\bar x\in [r,l]$ such that $f(\bar x)=\text{\rm sn}_k(r)$ and $f(t)<\text{\rm sn}_k(r)$ on $[r,\bar x)$.
Since $f(\bar x)\leq \text{\rm sn}_k(\bar x)$ by (2.1.1), there is an $a\in[0,1]$ such that
$f(\bar{x})=a\text{\rm sn}_k(\bar{x})$. We claim that
$$\int_{0}^{\bar{x}}a^m\text{\rm sn}_k^m(t) \,dt\geq\int_{0}^{r}\text{\rm sn}_k^m(t) \,dt. \eqno{(2.1)}$$
Note that $\frac{f(t)}{\text{\rm sn}_k(t)}\geq\frac{f(\bar x)}{\text{\rm sn}_k(\bar x)}=a$ on $(0,\bar{x}]$, which plus (2.1) implies that
$$\int_{0}^{\bar{x}}f^m(t)\geq \int_{0}^{\bar{x}}a^m\text{\rm sn}_k^m(t) \,dt\geq\int_{0}^{r}\text{\rm sn}_k^m(t) \,dt=\int_{0}^{x}f^m(t) \,dt.\eqno{(2.2)}$$
It then has to hold that $\bar{x}\geq x$; and thus $f(x)\leq\text{\rm sn}_k(r)$ by the choice of $\bar x$ and the equality holds only if
$\bar x=x$. Note that if $\bar x=x$, (2.2) plus `$\frac{f(t)}{\text{\rm sn}_k(t)}\geq a$' implies that $\frac{f(t)}{\text{\rm sn}_k(t)}\equiv a$ on  $(0,x]$; so it follows that $f(t)={\text{\rm sn}_k(t)}$ on $[0,x]$ because $D^+f(0)=1$ and thus $x=r$. This completes the proof of (2.1.2)
once we prove (2.1).

It then just remains to verify (2.1). We first consider the case where $k\leq 0$, or $k>0$ and $\bar x\leq\frac{\pi}{2\sqrt{k}}$.
In this case, for any $t_1\in [0, r]$ there is a unique $t_2\in [t_1, \bar x]$ such that $\text{\rm sn}_k^m(t_1)=a^m\text{\rm sn}_k^m(t_2)$.
By considering the inverse functions of $\text{\rm sn}_k^m(t)|_{[0,r]}$ (note that we have assumed that $r\leq\frac{\pi}{2\sqrt{k}}$ if $k>0$)
and $a^m\text{\rm sn}_k^m(t)|_{[0,\bar{x}]}$
which are both defined on $[0,\text{\rm sn}_k^m(r)]$, in order to see (2.1) we just need to show
$$\bar x-t_2\geq r-t_1.\eqno{(2.3)}$$
It is not hard to see that (2.3) will be true if we can show that
$$(\text{\rm sn}_k^m(t))^{'}|_{t_1}\geq(a^m\text{\rm sn}_k^m(t))^{'}|_{t_2} \text{ for all $t_1\in [0, r]$}.\eqno{(2.4)}$$
In fact, note that $\text{\rm sn}_k(t_1)=a\text{\rm sn}_k(t_2)$, so
\begin{align*} (\text{\rm sn}_k^m(t))^{'}|_{t_1}=& m\text{\rm sn}_k^{m-1}(t_1)\text{\rm sn}_k'(t_1),\\
(a^m\text{\rm sn}_k^m(t))^{'}|_{t_2}=& ma^m\text{\rm sn}_k^{m-1}(t_2)\text{\rm sn}_k'(t_2)=
m\text{\rm sn}_k^{m-1}(t_1)(a\text{\rm sn}_k'(t_2)).
\end{align*}
And by taking into account $a\in[0,1]$, we can check that $\text{\rm sn}_k'(t_1)\geq a\text{\rm sn}_k'(t_2)$ for $k>0, =0, <0$  one by one. Then (2.4) follows. As for the only remaining case where $k>0$ and $\bar{x}>\frac{\pi}{2\sqrt{k}}$,
note that it suffices to show
$$\int_{0}^{\frac{\pi}{\sqrt{k}}-\bar{x}}a^m\cdot\text{\rm sn}_k^m(t) \,dt\geq\int_{0}^{r}\text{\rm sn}_k^m(t) \,dt.\eqno{(2.5)}$$
Note that $a\text{\rm sn}_k(\frac{\pi}{\sqrt{k}}-\bar{x})=a\text{\rm sn}_k(\bar{x})=\text{\rm sn}_k(r)$ which implies $\frac{\pi}{\sqrt{k}}-\bar{x}\geq r$. Then by considering the inverse functions of $\text{\rm sn}_k^m(t)|_{[0,r]}$ and $a^m\text{\rm sn}_k^m(t)|_{[0,\frac{\pi}{\sqrt{k}}-\bar{x}]}$, we can similarly verify (2.5).

\vskip1mm

(2.1.3)\ \
Since $\frac{f(t)}{\text{\rm sn}_k(t)}$ is decreasing on $(0,l]$, it is clear that $D^\pm\left(\frac{f(x)}{\text{\rm sn}_k(x)}\right)\leq 0$, which implies
$$D^\pm f(x)\cdot\text{\rm sn}_k(x)-f(x)\cdot\text{\rm sn}'_k(x)\leq0, \text{ i.e. } D^\pm f(x)\leq\frac{f(x)}{\text{\rm sn}_k(x)}\text{\rm sn}'_k(x).$$
Note that $\frac{\text{\rm sn}'_k(x)}{\text{\rm sn}_k(x)}\leq \frac{\text{\rm sn}'_k(r)}{\text{\rm sn}_k(r)}$ (because $x\geq r$) and  $f(x)\leq\text{\rm sn}_k(r)$ by (2.1.2), then $\frac{f(x)}{\text{\rm sn}_k(x)}\text{\rm sn}'_k(x)\leq\text{\rm sn}'_k(r)$, and the equality implies $x=r$ by (2.1.2). This completes the proof of (2.1.3).

\vskip1mm

(2.1.4)\ \ Since $f$ is right continuous in $[0,l)$, it is easy to see that the function $x(r)$ in (2.1.1)
is right differentiable and satisfies
$$f^m(x(r))x'_+(r)=\text{\rm sn}_k^m(r).\eqno{(2.6)}$$
By (2.1.2), we have that $f(x(r))\leq\text{\rm sn}_k(r)$, so $x'_+(r)\geq1$. And if $x_{+}'(r_0)=1$, i.e. $f(x(r_0))=\text{\rm sn}_k(r_0)$,
then $x(r_0)=r_0$  by (2.1.2) and so $f(t)={\text{\rm sn}_k(t)}$ on $[0,r_0]$.
\end{proof}

\noindent{\bf $\bullet$\ \  A kind of examples satisfying Lemma 2.1}

\vskip2mm

In comparison geometry, many classical results involve continuous functions satisfying $f''(t)+kf(t)\leq 0$ in the support sense \footnote{At $t\in (a,b)$, we say that {\it $f^{\prime \prime}(t) \leq B$ in the support sense} if there is an $A$ such that
$f(t+\tau) \leq f(t)+A \tau+\frac12B \tau^{2}+o\left(\tau^{2}\right)$ (for other versions of the definition refer to \cite{Na} and \cite{Pe}).} ([CE], [Pe], cf. [PP], [Pet], [GP], [HSW]). In the following, we will show that such kind of functions are examples satisfying Lemma 2.1.
Note that $\text{\rm sn}_k(t)$ is just the solution of $f''(t)+kf(t)=0$ with $f(0)=0$ and $f'(0)=1$.

\begin{lemma} Let $f: [0,l]\to \mathbb{R}$ be a continuous function with $f(0)=0$, and let $k$ be a real number such that $l\leq\frac{\pi}{\sqrt k}$ if $k>0$. If
$f''(t)+kf(t)\leq 0$ \footnote{In the support sense, if $f''(t)+kf(t)\geq 0$, then $g(t)\triangleq-f(t)$ satisfies $g''(t)+kg(t)\leq 0$.} in the support sense for all $t\in (0,l)$, then
$$\dfrac{f(t)}{\text{\rm sn}_k(t)}, \text{ whose limit exists as $t\to 0^+$ (i.e. $f'_+(0)$ exists), is decreasing in $(0,l)$}; \eqno{(2.7)}$$
moreover, if $\frac{f(t_1)}{\text{\rm sn}_k(t_1)}=\frac{f(t_2)}{\text{\rm sn}_k(t_2)}$ for some $t_1<t_2$, then $f(t)\equiv f'_+(0){\rm sn}_k(t)$ on $(0,t_2]$.
\end{lemma}

We should point out that Lemma 2.2 (especially (2.7)) for smooth $f$ is known to experts. In the proof of Lemma 2.2, the main work is to show that $\frac{f(t)}{\text{\rm sn}_k(t)}$ satisfies a property that its upper right and left Dini derivatives are not positive.
It should be known that a function is decreasing if and only if it satisfies such a property
(this is almost obviously true for a continuous function, cf. [KK]).
This is crucial to Lemma 2.2 and we cannot find a reference for it, so we will formulate and prove it.

\begin{lemma} \label{lem:2.3}
Let $f: [a,b]\to \mathbb{R}$ be a function (maybe not continuous). Then $f$ is decreasing on $[a,b]$ if and only if $D^+f(x)\leq 0$ and $D^-f(x)\leq 0$ for all $x\in [a,b]$ \footnote{It needs only $D^+f(a)\leq 0$ and $D^-f(b)\leq 0$ at $a$ and $b$.}.
\end{lemma}

\begin{proof}
We just need to show the sufficiency. Suppose that $f$ is not decreasing. Then there exist $a_1,b_1\in [a,b]$ with $a_1<b_1$ and $f(a_1)<f(b_1)$, so $\delta\triangleq \frac{f(b_1)-f(a_1)}{b_1-a_1}>0$.
If $f(\frac{a_1+b_1}{2})\leq \frac{f(a_1)+f(b_1)}{2}$, then let $a_2=\frac{a_1+b_1}{2}$ and $b_2=b_1$; otherwise let $a_{2}=a_1$ and $b_{2}=\frac{a_1+b_1}{2}$. Note that $\frac{f(b_2)-f(a_2)}{b_2-a_2}\geq\delta$.
By repeating this process, we can obtain a sequence of intervals  $\{[a_i, b_i]\}_{i=1}^\infty$ with
$[a_{i+1}, b_{i+1}]\subset [a_i, b_i]$ and $b_i-a_i\to 0$ as $i\to \infty$, moreover each $\frac{f(b_i)-f(a_i)}{b_i-a_i}\geq \delta$.
Then consider the $x_0=\bigcap\limits_{i=1}^\infty[a_i,b_i]$. Note that  $\frac{f(b_i)-f(x_0)}{b_i-x_0}\geq\delta$ or $\frac{f(a_i)-f(x_0)}{a_i-x_0}\geq\delta$,
so $D^+f(x_0)\geq\delta$ or $D^-f(x_0)\geq\delta$, a contradiction.
\end{proof}

\begin{remark}\label{rem2.4} {\rm Let $f:[a,b]\to \mathbb{R}$ be a continuous function,
and $f''(t)\leq 0$ in the support sense for all $t\in (a,b)$. It is well known that
$f(t)$ is concave on $[a,b]$, and thus $f'_+(t)$ exists
and is decreasing on $[a,b)$ (ref. [Na], [PP]). By Lemma \ref{lem:2.3},
$$ D^+f'_+(t)\leq 0,\ D^-f'_+(t)\leq 0\eqno{(2.8)}$$
for $t\in [a,b)$ (at $a$, it just needs $ D^+f'_+(a)\leq 0$). Similarly, $f'_-(t)$
exists and is decreasing on $(a,b]$. Moreover, a basic fact (due to the concavity) is that, for all $t\in (a,b)$,
$$f'_-(t)\geq f'_+(t) \text{ and } \lim_{\tau \to t^-}f'_+(\tau)=f'_-(t).\eqno{(2.9)}$$
As a result, if $f'_-(t)>f'_+(t)$ for some $t\in (a,b)$, then
$$ D^-f'_+(t)=-\infty.\eqno{(2.10)}$$}
\end{remark}

\begin{proof} [Proof of Lemma 2.2]\

Put $g(t)\triangleq\frac{f(t)}{\sn_k(t)}$, $t\in (0,l)$, which is a continuous function.
For (2.7), we just need to show that the right derivative $f'_+(0)$ exists
(so $\lim\limits_{t\to 0^+}g(t)=f'_+(0)$ because $f(0)=0$), and that $g'_+(t)\leq 0$ on $(0,l)$.
For the rigidity part, it suffices to show that if $g'_+(t_0)=0$ with $t_0\in(0,l)$,
then $g'_+(t)\equiv 0$ on $(0,t_0]$.

For the purpose, consider a twice differentiable function $F(t)$ with $F''(t)=f(t)$. Note that
$f(t)+kF(t)$ satisfies $(f(t)+kF(t))''\leq 0$ in the support sense on $(0,l)$. Then,
$f'_+(t)$ exists for all $t\in [0,l)$ (see Remark \ref{rem2.4}), and by (2.8)
$$D^+f'_+(t)+kf(t)=D^+(f(t)+kF(t))'_+\leq 0,\quad D^-f'_+(t)+kf(t)=D^-(f(t)+kF(t))'_+\leq 0.\eqno{(2.11)}$$
Moreover, $f'_-(t)$ exists on $(0,l]$; and by the first inequality of (2.9) and (2.10), for all $t\in (0,l)$,
$$f'_-(t)\geq f'_+(t), \text{ and $D^-f'_+(t)=-\infty$ if $f'_-(t)>f'_+(t)$}.\eqno{(2.12)}$$
On the other hand, for all $t\in (0,l)$,
$$g'_+(t)=\frac{f'_+(t)\sn_k(t)-f(t)\sn'_k(t)}{\sn_k^2(t)}.$$
Let $h(t)\triangleq f'_+(t)\sn_k(t)-f(t)\sn'_k(t)$ for $t\in[0,l)$ (maybe not continuous) with $h(0)=0$. By the first inequality of (2.11),
it is easy to see that
$$D^+h(t)=\left(D^+f'_+(t)+kf(t)\right) \sn_k(t) \leq 0.$$
Meantime,
$$D^-h(t)=\left(D^-f'_+(t)+kf(t)\right) \sn_k(t)+\left(f'_+(t)-f'_-(t)\right)\sn'_k(t).$$
Plus (2.12) (it just needs the first part of (2.12) if $\sn'_k(t)\geq0$, i.e. if
$k\leq 0$ or if $k>0$ and $t\leq\frac{\pi}{2\sqrt k}$), the second inequality of (2.11) also guarantees $D^-h(t)\leq 0$.
Then, by Lemma \ref{lem:2.3}, $h(t)$ is decreasing on $[0,l)$; thus, $h(t)\leq h(0)=0$, and so
$g'_+(t)\leq 0$ on $(0,l)$.

We now assume $g'_+(t_0)=0$, i.e. $h(t_0)=0$, for some $t_0\in(0,l)$. Then, $h(t)|_{[0,t_0]}\equiv0$
because $h(t)$ is decreasing on $[0,l)$ with $h(0)=0$. It follows that $g'_+(t)\equiv 0$ on $(0,t_0]$.
\end{proof}

\vskip2mm

\noindent{\bf $\bullet$\ \  Corollaries of Lemma 2.2}

\vskip2mm

\begin{coro}\label{coro2.5} Let $f$ be the function in Lemma 2.2. Then, $f(t)\leq f'_+(0)\text{\rm sn}_k(t)$ on $[0,l]$;
and if $f(t_0)=f'_+(0)\text{\rm sn}_k(t_0)$ for some $t_0\in (0,l]$, then  $f(t)=f'_+(0)\text{\rm sn}_k(t)$
on $[0, t_0]$. Moreover,
$$\psi(t)\triangleq f(t)-f'_+(0)\text{\rm sn}_k(t) \text{ is decreasing on $[0,l]$},\eqno{(2.13)}$$
where it needs $l\leq\frac{\pi}{2\sqrt k}$ if $k>0$; and if $\psi(t_1)=\psi(t_2)$ for some $t_1<t_2$, then $\psi(t)\equiv 0$ on $[0,t_2]$.
\end{coro}

To (2.13), the condition `$l\leq\frac{\pi}{2\sqrt k}$' for $k>0$ is crucial. As a counterexample for $k=1$ and $l=\pi$, one can consider $f(t)\triangleq \frac12\sin(2t)$ on $[0,\frac\pi2]$  and $\cos t$ on $[\frac\pi2,\pi]$. We should point out that
Corollary 2.5 can be seen just by analyzing the difference function $f(t)-f'_+(0)\text{\rm sn}_k(t)$ without involving the quotient function in (2.7); and the first part of the corollary (the conclusions before (2.13)) is known to experts (cf. [PP], \cite{Pe}, [GP]).

\begin{proof}
Note that the first part (before (2.13)) of the corollary follows from (2.7) in Lemma 2.2 directly. Then, consider $\psi(t)\triangleq f(t)-f'_+(0)\text{\rm sn}_k(t)\leq 0$,
which satisfies $\psi''(t)+k\psi(t)\leq 0$ in the support sense on $(0,l)$ (with $\psi(0)=\psi'_+(0)=0$).
According to (2.7), for $0<t_1<t_2$, we have that
$$\frac{\psi(t_1)}{\text{\rm sn}_k(t_1)}\geq\frac{\psi(t_2)}{\text{\rm sn}_k(t_2)}, \text{ i.e. }
\psi(t_1)\geq\frac{\text{\rm sn}_k(t_1)}{\text{\rm sn}_k(t_2)}\psi(t_2).$$
This completes the proof because $\text{\rm sn}_k(t)$ is strictly increasing on
$[0,l]$ with $l\leq\frac{\pi}{2\sqrt k}$ for $k>0$.
\end{proof}

In the proof above, if $k\leq 0$, then `$\psi''(t)+k\psi(t)\leq 0$' implies $\psi''(t)\leq 0$ because $\psi(t)\leq 0$; so $\psi(t)$ is concave on $[0,l]$ with $\psi(0)=\psi'_+(0)=0$, and  the conclusion about $\psi(t)$ in Corollary 2.5 follows.

\begin{coro}\label{coro2.6} Let $f$ be the function in Lemma 2.2. Then the followings hold:

\noindent{\rm (1)} If $f(l)=0$, where $l<\frac{\pi}{\sqrt k}$ if $k>0$, then $f(t)\geq 0$ for all $t\in[0,l]$.

\noindent{\rm (2)} If $f(t_0)=f'_+(0)\text{\rm sn}_k(t_0)$ for some $t_0\in[0,l]$, then $f(t)=f'_+(0)\text{\rm sn}_k(t)$ on $[0,t_0]$.

\noindent{\rm (3)} If $k>0$ and $f(t_0)=0$ with $t_0<\frac{\pi}{\sqrt k}$, then either $f(t)|_{[0,t_0]}\equiv0$,
or $f(t)|_{(0,t_0)}>0$ and the maximum of $f(t)|_{(0,t_0)}$ is achieved at $\bar t\leq\frac{\pi}{2\sqrt k}$, and if $\bar t=\frac{\pi}{2\sqrt k}$ then $f(t)=f'_+(0)\text{\rm sn}_k(t)$ on $[0,\bar t]$.

\noindent{\rm (4)} If $k>0$ and $f'_+(0)<0$, then $f(t)$ is strictly decreasing on $[0,\frac{\pi}{2\sqrt k}]$.
\end{coro}

In Corollary 2.6, conclusions (1), (2) except for the case
`$k>0$ and $t_0=\frac{\pi}{\sqrt k}$', and (3) except for the rigidity part for $\bar t=\frac{\pi}{2\sqrt k}$
can be obtained by analyzing the function $f(t)-f'_+(0)\text{\rm sn}_k(t)$
(ref. [HSW] for (1), [PP] for (2) where $t_0<\frac{\pi}{\sqrt k}$ if $k>0$, [GP] for the conclusion `$\bar t\leq\frac{\pi}{2\sqrt k}$' in (3)).

\begin{proof} (1) is an immediate corollary of (2.7) in Lemma 2.2.

(2) is an immediate corollary of (2.7) except for the case where
$k>0$ and $t_0=\frac{\pi}{\sqrt k}$. In the exceptional case, consider $g(t)\triangleq f(\frac{\pi}{\sqrt k}-t)$, which also
satisfies $g''(t)+kg(t)\leq 0$  in the support sense for $t\in (0,l)$. By applying (1) to $f(t)$ and $g(t)$ simultaneously, we can conclude that
either $f(t)\geq 0$ on $[0,\frac{\pi}{\sqrt k}]$, or $f(t)\leq 0$ on $[0,\frac{\pi}{\sqrt k}]$. Then by applying (2.7) to $f(t)$ and $g(t)$
simultaneously, one can see that $f(t)=f'_+(0)\text{\rm sn}_k(t)$ for all $t\in[0,\frac{\pi}{\sqrt k}]$.

For (3), note that $f(t)\geq 0$ on $(0,t_0)$ (by (1)). So, on $(0,t_0)$, $f''(t)\leq 0$ in the support sense, i.e. $f(t)$ is concave (see Remark 2.4).
Thus, $f(t)|_{[0,t_0]}\equiv0$, or $f(t)|_{(0,t_0)}>0$ with $f'_+(0)>0$. In the latter case, assume that $f(\bar t)$
is the maximum of $f(t)|_{[0,t_0]}$. (2.7) guarantees $\bar t\leq\frac{\pi}{2\sqrt k}$ because
$\text{\rm sn}_k(t)$ is strictly decreasing on $[\frac{\pi}{2\sqrt k},\frac{\pi}{\sqrt k}]$. If $\bar t=\frac{\pi}{2\sqrt k}$,
set $\bar f(t)\triangleq f(t)$ on $[0,\frac{\pi}{2\sqrt k}]$ and $f(\frac{\pi}{\sqrt k}-t)$ on
$[\frac{\pi}{2\sqrt k}, \frac{\pi}{\sqrt k}]$.
Since $f(\frac{\pi}{2\sqrt k})$ is the maximum of $f(t)|_{[0,t_0]}$, $\bar f(t)$ also satisfies $\bar f''(t)+k\bar f(t)\leq 0$  in the support sense for all $t\in (0,\frac{\pi}{\sqrt k})$.
Then, by (2.7) for $\bar f$, it has to hold that $\bar f(t)=f'_+(0)\text{\rm sn}_k(t)$ on $t\in [0,\bar t]$.

(4) follows from (2.7) directly (note that $\text{\rm sn}_k(t)$ is strictly increasing on $[0,\frac{\pi}{2\sqrt k}]$).
\end{proof}

\begin{coro}\label{coro2.7} Let $f$ be the function in Lemma 2.2. Suppose that $f'_+(0)=1$ and $f(t)|_{(0,l)}>0$, and
for a given positive integer $m$, $\int_0^{x}f^m(t)dt=\int_0^{r}{\text{\rm sn}_k^m(t)}dt$ with $x,r\in(0,l)$.
If $k\geq 0$ or $f(t)=\text{\rm sn}_{\bar k}(t)$ with $\bar k>k$, then both $\frac{f(x)}{\text{\rm sn}_k(r)}$ and $\frac{r}{x}$ are less than or equal to $1$ and decreasing with respect to $r$.
\end{coro}

When $k$ is negative and $f(t)\neq\text{\rm sn}_{\bar k}(t)$, the conclusion of Corollary 2.7 might not be true. As a counterexample for $k=-1$ and $m=1$, we can consider the function
$$f(t)=\sinh t+\int_0^t(\cos x-1)\sinh(t-x)dx,\ t\in[0,5].$$

\begin{proof} Note that $f$ satisfies the conditions of Lemma 2.1, so $\frac{f(x)}{\text{\rm sn}_k(r)}\leq 1$ and $\frac{r}{x}\leq1$ by (2.1.1) and (2.1.2). And here $f$ is in addition continuous, so similar to (2.6) we have that $$f^m(x)x'(r)=\text{\rm sn}_k^m(r)$$
and thus $x'(r)\geq1$ is also continuous with respect to $r$.
It is easy to check that
$$\left(\frac{f^m(x)}{\text{\rm sn}_k^m(r)}\right)'_+=-\frac{(x')'_+(r)}{(x'(r))^2},$$
and note that
$$\left(\frac{r}{x}\right)'=\frac{x-rx'}{x^2} \text{ with }  \lim\limits_{r\to 0^+}(x-rx')=0  \text{ and }  \left(x-rx'\right)'_+=-r(x')'_+.$$
It then suffices to show that $(x')'_+(r)\geq 0$ for $r>0$. Note that
\begin{align*} (x')'_+(r)=&\frac{mf^m(x)\text{\rm sn}_k^{m-1}(r)\text{\rm sn}_k'(r)-m\text{\rm sn}_k^m(r)f^{m-1}(x)f'_+(x)x'(r)}{f^{2m}(x)}\\
=&\frac{m\text{\rm sn}_k^{m-1}(r)}{f^{2m+1}(x)}\left(f^{m+1}(x)\text{\rm sn}_k'(r)-\text{\rm sn}_k^{m+1}(r)f'_+(x)\right).
\end{align*}
 Then by Lemma 2.3, it suffices to show that
$$D_\pm\left(f^{m+1}(x)\text{\rm sn}_k'(r)-\text{\rm sn}_k^{m+1}(r)f'_+(x)\right)\geq 0$$
with respect to $r$. Due to the similarity, we consider only the cases where $k=0,1,-1$.

$\bullet$\ \ $k=0$: In this case, $D^+f'_+\leq 0$ and $D^-f'_+\leq 0$ by (2.8), then
\begin{align*} D_+\left(f^{m+1}(x)-r^{m+1}f'_+(x)\right)=-r^{m+1}D^+f'_+(x)x'(r)\geq0,
\end{align*}
and by taking into account (2.9) we have that
\begin{align*} D_-\left(f^{m+1}(x)-r^{m+1}f'_+(x)\right)=(m+1)r^{m}\left(f'_-(x)-f'_+(x)\right)-r^{m+1}D^-f'_+(x)x'(r)\geq0.
\end{align*}

$\bullet$\ \ $k=1$: In this case, $D^+f'_+\leq -f$ and $D^-f'_+\leq -f$ by (2.11), then
\begin{align*} &D_+\left(f^{m+1}(x)\cos r-\sin^{m+1}rf'_+(x)\right)\\
=&-f^{m+1}(x)\sin r-\sin^{m+1}rD^+f'_+(x)x'(r)\\
\geq & -f^{m+1}(x)\sin r+\sin^{m+1}rf(x)x'(r)\\
\geq & -f^{m+1}(x)\sin r+\sin^{m+1}rf(x)\quad (\text{by } (2.1.4))\\
\geq & 0 \quad (\text{by } (2.1.2)).
\end{align*}
And similarly,
\begin{align*} &D_-\left(f^{m+1}(x)\cos r-\sin^{m+1}rf'_+(x)\right)\\
=&(m+1)\sin^{m}r\cos r\left(f'_-(x)-f'_+(x)\right)-f^{m+1}(x)\sin r-\sin^{m+1}rD^-f'_+(x)x'(r)\geq 0
\end{align*}
(note that $D^-f'_+(x)=-\infty$ when $f'_-(x)>f'_+(x)$, see (2.10)).

$\bullet$\ \ $k=-1$: In this case, our assumption is that $f(t)=\text{\rm sn}_{\bar k}(t)$ with $\bar k>-1$, then
$$D_\pm\left(f^{m+1}(x)\text{\rm cosh}(r)-\text{\rm sinh}^{m+1}(r)f'_+(x)\right)=\text{\rm sinh}(r)\left(\text{\rm sn}_{\bar k}^{m+1}(x)-\text{\rm sinh}^m(r)\text{\rm sn}_{\bar k}''(x)x'(r)\right).$$
It is clearly positive when $\bar{k}\geq0$ because $\text{\rm sn}_{\bar k}''(x)\leq0$ (note that
$x<\frac{\pi}{\sqrt{\bar k}}$ when $\bar k>0$ because $f(t)|_{(0,l)}>0$).
When  $\bar{k}\in(-1,0)$, note that $\text{\rm sn}_{\bar k}''(x)=(-\bar k)\text{\rm sn}_{\bar k}(x)$ and $x'(r)=\frac{\text{\rm sinh}^m(r)}{\text{\rm sn}_{\bar k}^m(x)}$, so it suffices to show that
$$\text{\rm sn}_{\bar k}^m(x)>\sqrt{-\bar k}\sinh^m(r). \eqno{(2.14)}$$
Note that
\begin{align*}
    \int_{0}^{r} \text{\rm sinh}^m(t) \,dt&=\int_{0}^{x} \text{\rm sn}_{\bar k}^m(t) \,dt\\
    &=\int_{0}^{x} \left(\frac{1}{\sqrt{-\bar{k}}}\right)^{m}\text{\rm sinh}^m(\sqrt{-\bar{k}}t) \,dt\\
    &=\int_{0}^{\sqrt{-\bar{k}}x} \left(\frac{1}{\sqrt{-\bar{k}}}\right)^{m+1}\text{\rm sinh}^m(t) \,dt,
\end{align*}
so we can get that $\sqrt{-\bar{k}}x<r$ because $0<-\bar{k}<1$. Then we can apply the method used in proving (2.1)
to derive that $\left(\frac{1}{\sqrt{-\bar{k}}}\right)^{m+1}\text{\rm sinh}^m(\sqrt{-\bar{k}}x)>\text{\rm sinh}^m(r)$, i.e. (2.14) holds.
If it is not true, then there is $r_1\in (\sqrt{-\bar{k}}x,r]$ such that $\left(\frac{1}{\sqrt{-\bar{k}}}\right)^{m+1}\text{\rm sinh}^m(\sqrt{-\bar{k}}x)=\text{\rm sinh}^m(r_1)$. And for any $t_1\in [0,\sqrt{-\bar{k}}x]$, there is a unique $t_2\in (t_1,r_1]$
such that $\left(\frac{1}{\sqrt{-\bar{k}}}\right)^{m+1}\text{\rm sinh}^m(t_1)=\text{\rm sinh}^m(t_2)$.
Note that
\begin{align*}
\left.\left[\left(\frac{1}{\sqrt{-\bar{k}}}\right)^{m+1}\text{\rm sinh}^m(t)\right]'\right|_{t_1}
&=m\left(\frac{1}{\sqrt{-\bar{k}}}\right)^{m+1}\text{\rm sinh}^{m-1}(t_1)\sqrt{1+\text{\rm sinh}^2(t_1)}\\
&=m\text{\rm sinh}^{m}(t_2)\sqrt{1+\frac1{\text{\rm sinh}^2(t_1)}},\\
\left(\text{\rm sinh}^m(t)\right)^{'}|_{t_2}&=m\text{\rm sinh}^{m}(t_2)\sqrt{1+\frac1{\text{\rm sinh}^2(t_2)}},
\end{align*}
so $\left.\left[\left(\frac{1}{\sqrt{-\bar{k}}}\right)^{m+1}\text{\rm sinh}^m(t)\right]'\right|_{t_1}>(\text{\rm sinh}^m(t))^{'}|_{t_2}.$
Then by considering the inverse functions, defined on $[0,\text{\rm sinh}^m(r_1)]$, of $\left.\left(\frac{1}{\sqrt{-\bar{k}}}\right)^{m+1}\text{\rm sinh}^m(t)\right|_{[0,\sqrt{-\bar{k}}x]}$ and $\text{\rm sinh}^m(t)|_{[0,r_1]}$,
we can conclude that
$$\int_{0}^{\sqrt{-\bar{k}}x} \left(\frac{1}{\sqrt{-\bar{k}}}\right)^{m+1}\text{\rm sinh}^m(t) \,dt< \int_{0}^{r_1} \text{\rm sinh}^m(t) \,dt\leq
\int_{0}^{r} \text{\rm sinh}^m(t),$$
a contradiction. That is, (2.14) has been verified, so the proof is completed.
\end{proof}


\section{Proofs of Theorems A and C}

Let $M$ be the manifold in Theorem A, and $(\rho, \theta)$ be the polar coordinates of $T_pM$, where $\theta$ is the canonical coordinates of the unit sphere $\Bbb S^{n-1}$.
Recall that $(\exp_p^{-1})^*g=d\rho^2+G(\rho,\theta)d\theta^2$ on the interior part of the segment domain of $p$, where $g$ is the metric of $M$ and $G(\rho,\theta)$ denotes $(n-1)\times(n-1)$ matrix.
Set $\lambda(\rho,\theta)\triangleq\sqrt{\det(G(\rho,\theta))}$. A key observation is that ``Ric$_M\geq (n-1)k$'' guarantees that  ([Pe], [WSY])
$$\frac{\lambda(\rho,\theta)}{\text{\rm sn}_k^{n-1}(\rho)}, \text{ whose limit is equal to $1$ as $\rho\to 0^+$, is decreasing with respect to } \rho. \eqno{(3.1)}$$
Furthermore, we can define
$\bar\lambda(\rho,\theta)\triangleq\begin{cases}\lambda(\rho,\theta), & \rho<\rho_\theta\\ 0, & \rho\geq\rho_\theta\end{cases}$ on the whole $T_pM$,
where $\rho_\theta$ satisfies that $\exp_p((\rho,\theta))|_{\rho\in[0,\rho_0]}$ is a segment for $\rho_0\leq\rho_\theta$, but not for $\rho_0>\rho_\theta$.
It follows that
$$\frac{\bar\lambda(\rho,\theta)}{\text{\rm sn}_k^{n-1}(\rho)}, \text{ whose limit is equal to $1$ as $\rho\to 0^+$, is still decreasing with respect to } \rho, \eqno{(3.2)}$$
where $\rho\leq\frac{\pi}{\sqrt{k}}$ if $k>0$. In the Lebesgue sense,
$$\text{Vol$(\partial B(p,\rho))\triangleq\int_{\Bbb S^{n-1}(1)} \bar\lambda(\rho,\theta)d\theta$,\ \
Vol$(B(p, r))=\int_0^r\int_{\Bbb S^{n-1}(1)} \bar\lambda(\rho,\theta)d\theta d\rho$}\ \footnote{From (3.2), it is
clear that Vol$(\partial B(p, \rho))\leq $Vol$(\partial B(\tilde p, \rho))$, and it is  not hard to see (1.1) in Theorem 1.1 (i.e. the relative volume comparison theorem, [Pe], [WSY]). }
$$  ([Pe], [WSY]). Note that $\bar\lambda(\rho,\theta)$ maybe is not continuous at $\rho_\theta$ with respect to $\rho$, nor is Vol$(\partial B(p,\rho))$, for example, at $\rho=\frac{\pi}{2}$ when $M$ is isometric to $\Bbb R\Bbb P^n$ with canonical metric.

\begin{proof}[Proof of Theorem A]\

Note that Vol$(\partial B(\tilde p,\rho))=v_1\cdot\text{\rm sn}_k^{n-1}(\rho)$, where $v_1$=Vol$(\Bbb S^{n-1}(1))$. Then by Lemma 2.1, it suffices to show that, with respect to $\rho$,
\vskip1mm
\noindent(3.3)\ \ Vol$(\partial B(p,\rho))$ is right continuous;
\vskip1mm
\noindent(3.4)\ \ Vol$'_+(\partial B(p,\rho))$ exists;
\vskip1mm
\noindent(3.5)\ \ $\frac{\text{Vol}(\partial B(p,\rho))}{\text{Vol}(\partial B(\tilde p,\rho))}, \text{ whose limit is equal to $1$ as $\rho\to 0^+$, is decreasing}$
\footnote{Plus (3.4), (3.5) implies that $\text{Vol$_+'(\partial B(p, \rho))\leq\frac{{\rm Vol}(\partial B(p, \rho))}{{\rm Vol}(\partial B(\tilde p, \rho))}$Vol$'(\partial B(\tilde p, \rho))\ (\leq$Vol$'(\partial B(\tilde p, \rho)) \text{ by Footnote 6})$}$.}.
\vskip1mm

Note that $\bar\lambda(\rho,\theta)$ is right continuous with respect to $\rho$, and $0\leq\bar\lambda(\rho,\theta)\leq \text{\rm sn}_k^{n-1}(\rho)$ by (3.2).
Then (3.3) is guaranteed by the Dominated Convergence Theorem. As for (3.4), note that when $\rho<\rho_\theta$,
$$\partial_\rho\lambda(\rho,\theta)=\triangle \rho\cdot\lambda(\rho,\theta)\eqno{(3.6)}$$
where $\triangle \rho$ is the Laplacian of the distance function $\rho$ to $p$ on $M$ (see the proof of Lemma 7.1.2  in [Pe]).
This implies that, for each $\rho$, there is a $C_\rho$ such that $|\bar\lambda'_+(\rho,\theta)|\leq C_\rho$
(note that on $\overline{B(p,\rho)}$ the sectional curvatures have both upper and lower bound). Then (3.4) is also guaranteed by the Dominated Convergence Theorem
and $$\text{Vol}'_+(\partial B(p,\rho))=\int_{\Bbb S^{n-1}(1)} \bar\lambda'_+(\rho,\theta)d\theta$$
(but Vol$'_-(\partial B(p,\rho))$ might not exist).

(3.5) is stated in [WSY] without proof, so for the completeness we will give a detailed proof for it.
Set $f(\rho)\triangleq\frac{\text{Vol}(\partial B(p,\rho))}{\text{Vol}(\partial B(\tilde p,\rho))}$,
whose limit is equal to $1$ as $\rho\to 0^+$ by (3.1). By Lemma 2.3, to see (3.5) it suffices to show that $D^{\pm}f(\rho)\leq0$ for all $\rho>0$. By a direct computation,
\begin{align*}
 D^{\pm}f(\rho) & =\varlimsup_{h \to 0^\pm} \frac{f(\rho+h)-f(\rho)}{h} \\
 & =\varlimsup_{h \to 0^\pm} \frac{1}{h}\left(\frac{\int_{\Bbb S^{n-1}(1)} \bar\lambda(\rho+h,\theta)d\theta}{v_1\cdot\text{\rm sn}_k^{n-1}(\rho+h)}-\frac{\int_{\Bbb S^{n-1}(1)} \bar\lambda(\rho,\theta)d\theta}{v_1\cdot\text{\rm sn}_k^{n-1}(\rho)}\right) \\
& =\varlimsup_{h \to 0^\pm} \frac{1}{v_1}\int_{\Bbb S^{n-1}(1)}\frac{1}{h}\left(\frac{ \bar\lambda(\rho+h,\theta)}{\text{\rm sn}_k^{n-1}(\rho+h)}-\frac{ \bar\lambda(\rho,\theta)}{\text{\rm sn}_k^{n-1}(\rho)}\right)d\theta.
 \end{align*}
Observe that, for each $\rho$ and $h\neq0$, $\frac{1}{h}\left(\frac{ \bar\lambda(\rho+h,\theta)}{\text{\rm sn}_k^{n-1}(\rho+h)}-\frac{ \bar\lambda(\rho,\theta)}{\text{\rm sn}_k^{n-1}(\rho)}\right)\leq 0$ by (3.2), and as a function with respect to $\theta$ it is measurable on $\Bbb S^{n-1}(1)$. It then follows from Fatou's lemma that
$$ D^{\pm}f(\rho)\leq \frac{1}{v_1}\int_{\Bbb S^{n-1}(1)}\varlimsup_{h \to 0^\pm}\frac{1}{h}\left(\frac{ \bar\lambda(\rho+h,\theta)}{\text{\rm sn}_k^{n-1}(\rho+h)}-\frac{ \bar\lambda(\rho,\theta)}{\text{\rm sn}_k^{n-1}(\rho)}\right)d\theta\leq 0;
$$
and thus (3.5) follows. This completes the proof of Theorem A.
\end{proof}

\begin{remark}\label{rem3.1} {\rm From (3.5), it is easy to
see that Vol$(\partial B(p,\rho))$ is differentiable almost everywhere (note that it is smooth when $\rho$ is less than the injective radius of $M$ at $p$). }
\end{remark}

Inspired by Corollary 2.7, a natural question is, in Theorem A, whether $\frac{\text{Vol}(\partial B(p,\bar r))}{\text{Vol}(\partial B(\tilde p, r))}$ and $\frac{r}{\bar r}$ are decreasing with respect to $r$. Theorem C just aims to answer it. For the purpose, we need more detailed properties of $\lambda(\rho,\theta)$ than in (3.1).
Note that (3.6) plus `$\triangle \rho\leq (n-1)\frac{\text{\rm sn}_k'(\rho)}{\text{\rm sn}_k(\rho)}$'
(the Laplace comparison theorem) implies
$\frac{\partial_\rho\lambda(\rho,\theta)}{\lambda(\rho,\theta)}\leq(n-1)\frac{\text{\rm sn}_k'(\rho)}{\text{\rm sn}_k(\rho)}$
which is equivalent to the monotonicity in (3.1),
while `$\triangle\rho\leq (n-1)\frac{\text{\rm sn}_k'(\rho)}{\text{\rm sn}_k(\rho)}$' is due to
$$\partial_\rho\triangle \rho+\frac{(\triangle \rho)^2}{n-1}\leq \partial_\rho\triangle \rho+|\text{Hess}(\rho)|^2=-\text{Ric} (\frac{\partial}{\partial\rho},\frac{\partial}{\partial\rho})\leq-(n-1)k\eqno{(3.7)}$$
(cf. Proposition 7.1.1 and the proof of Lemma 7.1.2  in [Pe]). And if the first equality in (3.7) holds, then, restricted to the $n-1$ dimensional space orthogonal to $\frac{\partial}{\partial \rho}$,
$$\text{Hess}(\rho)=\frac{\triangle \rho}{n-1}\cdot G(\rho,\theta)d\theta^2\ \text{(cf. the proof of Theorem 7.2.5 in [Pe])}.\eqno{(3.8)}$$
From (3.6) and (3.7), we can derive that $$\partial^2_\rho\sqrt[n-1]{\lambda(\rho,\theta)}=\frac{\sqrt[n-1]{\lambda(\rho,\theta)}}{n-1}
\left(\partial_\rho\triangle \rho+\frac{(\triangle \rho)^2}{n-1}\right)\leq-k\cdot\sqrt[n-1]{\lambda(\rho,\theta)},$$
i.e.,
$$\partial^2_\rho\sqrt[n-1]{\lambda(\rho,\theta)}+k\cdot\sqrt[n-1]{\lambda(\rho,\theta)}\leq 0,\eqno{(3.9)}$$
and the equality implies $\lambda(\rho,\theta)=\text{\rm sn}^{n-1}_k(\rho)$
\footnote{In the case where $\lambda(\rho,\theta)=\text{\rm sn}^{n-1}_k(\rho)$, $\triangle\rho=(n-1)\frac{\text{\rm sn}_k'(\rho)}{\text{\rm sn}_k(\rho)}$ by (3.6), and both of the two equalities in (3.7) hold; and so (3.8) can be rewritten as
$\text{Hess}(\rho)=\frac{\text{\rm sn}_k'(\rho)}{\text{\rm sn}_k(\rho)}\cdot g_\rho, \text{ where }g_\rho=G(\rho,\theta)d\theta^2$.
Then it can be checked that
$$\nabla_{\partial\rho}\text{Hess}(\rho)+\text{Hess}^2(\rho)=-k\cdot g_\rho \text{ with } \lim_{\rho\to 0}\text{Hess}(\rho)=0.$$
This implies $g=d\rho^2+G(\rho,\theta)d\theta^2=d\rho^2+ \text{\rm sn}^{n-1}_k(\rho)d\theta^2$
according to Brinkmann's Theorem  (Theorem 4.3.3  in [Pe]), and thus
$B(p,\rho)$ is isometric to $B(\tilde p, \rho)\subset\Bbb S^n_k$ (cf. the proof of the Maximal Diameter Theorem in [Pe]).
Note that, in Theorem 1.1, if {\rm Vol}$(B(p, \rho_0))=${\rm Vol}$(B(\tilde p, \rho_0))$, then $\lambda(\rho,\theta)=\text{\rm sn}^{n-1}_k(\rho)$ for all $\rho<\rho_0$ by (3.1),
and thus $B(p,\rho_0)$ is isometric to $B(\tilde p, \rho_0)\subset\Bbb S^n_k$ (this is not mentioned in [Pe] and [WSY]).}
(note that $\lim\limits_{\rho\to0^+}\frac{\lambda(\rho,\theta)}{\text{\rm sn}^{n-1}_k(\rho)}=1$).

We are now ready to prove Theorem C.

\begin{proof}[Proof of Theorem C]\

Note that condition (C1) plus (3.9) enables us to apply Corollary 2.7 to see (1.2) (note that (3.9) might not be true if $\lambda(\rho,\theta)$ is replaced with $\bar\lambda(\rho,\theta)$). And under condition (C2), we have $\lambda(\rho,\theta)=\text{\rm sn}^{n-1}_{\bar k}(\rho)$, so we can also apply Corollary 2.7 to see (1.2).
\end{proof}

\begin{remark}\label{rem3.2} {\rm Note that we cannot derive that
$$\partial^2_\rho\sqrt[n-1]{\int_{\Bbb S^{n-1}(1)}\lambda(\rho,\theta)d\theta}+k\cdot\sqrt[n-1]{\int_{\Bbb S^{n-1}(1)}\lambda(\rho,\theta)d\theta}\leq 0$$
from (3.9) except when $n=2$. This is why we consider only the case `$n=2$' in condition (C1). }
\end{remark}

\section{An application of Lemma 2.2}

As an application of Lemma 2.2, we will give a relative version of Toponogov's Theorem (see Remark 4.1 below). In a complete Riemannian manifold $M$ with sectional curvature $\sec_M\geq k$, let $[pq]$ be a minimal geodesic with $|pq|<\frac\pi{\sqrt k}$ for $k>0$ and $\gamma(t)|_{t\in[0,l]}$ be an arc-length parameterized geodesic with $\gamma(0)=q$ and $l<\frac\pi{\sqrt k}$ for $k>0$; and in
$\Bbb S^2_k$, and let $[\tilde p\tilde q]$ be a minimal geodesic with $|\tilde p\tilde q|=|pq|$ and $\tilde\gamma(t)|_{t\in[0,l]}$ be an arc-length parameterized geodesic with $\tilde\gamma(0)=\tilde q$. Recall that Toponogov's Theorem asserts that {\it if $\angle pq\gamma(l)=\angle\tilde p\tilde q\tilde\gamma(l)$, then $|p\gamma(t)|\leq|\tilde p\tilde\gamma(t)|$.}

By the second variation formula, it is well known that
$\varphi_k''(|\tilde p\tilde\gamma(t)|)+k\varphi_k(|\tilde p\tilde\gamma(t)|)=1$ for all $t\in (0,l)$, and $\varphi_k''(|p\gamma(t)|)+k\varphi_k(|p\gamma(t)|)\leq1$ in the support sense (note that $\gamma(t)|_{[0,l]}$ might contain the point of the cut locus of $p$) because $\sec_M\geq k$, where the function $\varphi_k$ is defined by
$$\varphi_k(\rho)=\begin{cases}\frac{1}{k}(1-\cos(\sqrt{k} \rho)), & k>0\\ \frac12\rho^2, & k=0\\ \frac{1}{-k}(\cosh(\sqrt{-k}\rho)-1), & k<0\end{cases}.$$
Hence, the function $f(t)\triangleq \varphi_k(|p\gamma(t)|)-\varphi_k(|\tilde p\tilde \gamma(t)|)$ satisfies
$$f''(t)+kf(t)\leq 0 \text{ in the support sense for $t\in (0,l)$  (cf. [Pe], [PP])}.\eqno{(4.1)}$$
And note that $f'_+(0)\leq 0$ by the first variation formula. Then by Lemma 2.2 (or Corollary 2.5), we get that
$f(t)\leq 0$ for all $t$, which is equivalent to $|p\gamma(t)|\leq|\tilde p\tilde\gamma(t)|$ (i.e. Toponogov's Theorem follows).

Let $\bar\gamma(t)|_{t\in[0,l]}$ be another arc-length parameterized geodesic with $\bar\gamma(0)=\tilde q$ in $\Bbb S^2_k$. From above, Toponogov's Theorem can be formulated as
$$\varphi_k(|p\gamma(t)|)-\varphi_k(|\tilde p\bar\gamma(t)|)\leq \varphi_k(|\tilde p\tilde \gamma(t)|)-\varphi_k(|\tilde p\bar \gamma(t)|).\eqno{(4.2)}$$
Note that $g(t)\triangleq \varphi_k(|p\gamma(t)|)-\varphi_k(|\tilde p\bar \gamma(t)|)$ also satisfies
$g''(t)+kg(t)\leq 0$ in the support sense for $t\in (0,l)$; and by the Law of Cosine,
$\varphi_k(|\tilde p\tilde \gamma(t)|)-\varphi_k(|\tilde p\bar\gamma(t)|)=(\cos\angle\tilde p\tilde q\bar\gamma(l)-\cos\angle pq\gamma(l))\sn_k(|pq|)\sn_k(t)$. Then by Lemma 2.2 plus the first variation formula, (4.2) (or say Toponogov's Theorem) has a relative version as follows.

\begin{remark}\label{rem4.1} {\rm In the situation of (4.2), if $\angle \tilde p\tilde q\bar\gamma(l)<\angle pq\gamma(l)$
(resp. $\angle \tilde p\tilde q\bar\gamma(l)>\angle pq\gamma(l)$), then
$$\frac{\varphi_k(|p\gamma(t)|)-\varphi_k(|\tilde p\bar\gamma(t)|)}{\varphi_k(|\tilde p\tilde \gamma(t)|)-\varphi_k(|\tilde p\bar \gamma(t)|)}
\text{ is decreasing (resp. increasing) with respect to $t$}$$
and has limit $\geq 1$ (resp. $\leq 1$) as $t\to 0^+$. Or uniformly,
$$\Phi(t)\triangleq\frac{\varphi_k(|p\gamma(t)|)-\varphi_k(|\tilde p\bar\gamma(t)|)}{\sn_k(t)}\text{ is decreasing with respect to $t$};$$
moreover, if $\Phi(t_1)=\Phi(t_2)$ for some $t_1,t_2\in (0,l]$ with $t_1<t_2$,
then $\Phi(t)\equiv\Phi(t_2)$  on $(0,t_2]$.}
\end{remark}

\begin{remark}\label{rem4.2} {\rm In the situation of (4.1), $f(t)$
satisfies $f'_+(0)\leq 0$. Then by Corollary 2.5 we have that
$$f(t) \text{ is decreasing  with respect to $t$},$$
where it needs $t\leq\frac{\pi}{2\sqrt k}$ if $k>0$; and if $f(t_1)=f(t_2)$ with $t_1<t_2$
and $t_2\leq\frac{\pi}{2\sqrt k}$ for $k>0$, then $f(t)\equiv0$ ($\Leftrightarrow\ |p\gamma(t)|=|\tilde p\tilde\gamma(t)|$) on $[0,t_2]$.}
\end{remark}

Compared to Remarks 4.1 and 4.2, one cannot expect the decreasing property of $\frac{\varphi_k(|p\gamma(t)|)}{\varphi_k(|\tilde p\tilde\gamma(t)|)}$.

\begin{remark}\label{rem5.3} {\rm The arguments in this section also apply to the case where $M$ is a complete and geodesic Alexandrov space with curvature $\geq k$.}
\end{remark}


\vskip6mm

\noindent School of Mathematical Sciences (and Lab. math. Com.
Sys.), Beijing Normal University, Beijing, 100875,
People's Republic of China

\noindent E-mail: suxiaole$@$bnu.edu.cn; 202231130006@mail.bnu.edu.cn; wyusheng$@$bnu.edu.cn

\end{document}